\documentclass[12pt]{article}
\usepackage{anysize,amsmath,amssymb,amsthm,url}
\newtheorem{theorem}{Theorem}[section]

\newtheorem{cor}[theorem]{Corollary}
\newtheorem{preconj}[theorem]{Conjecture}
\newenvironment{conj}{\preconj\rm}{\endpreconj}

\newcommand{\End}{\mathop{\mathrm{End}}}
\newcommand{\Aut}{\mathop{\mathrm{Aut}}}
\renewcommand\le{\leqslant}
\renewcommand\ge{\geqslant}

\begin{document}
\title{Synchronization and separation\\ in the Johnson schemes}
\author{Mohammed Aljohani${}^{a}$, John Bamberg${}^{b}$ and
Peter J. Cameron${}^{a}$\footnote{Corresponding author:
\texttt{pjc20@st-andrews.ac.uk}}}
\date{\small
${}^{a}$ School of Mathematics and Statistics, University of St Andrews,
North Haugh, St Andrews, Fife KY16 9SS, U.K.\\
${}^{b}$ School of Mathematics and Statistics, University of Western Australia,
Crawley, Perth WA 6009, Australia}
\maketitle

\begin{abstract}
Recently Peter Keevash solved asymptotically the existence question for
Steiner systems by showing that $S(t,k,n)$ exists whenever the necessary
divisibility conditions on the parameters are satisfied and $n$ is sufficiently
large in terms of $k$ and $t$. The purpose of this paper is to make a
conjecture which if true would be a significant extension of Keevash's
theorem, and to give some theoretical and computational evidence for the
conjecture.

We phrase the conjecture in terms of the notions (which we define here) of
synchronization and separation for association schemes. These definitions are
based on those for permutation groups which grow out of the theory of
synchronization in finite automata. In this theory, two classes of permutation
groups (called \emph{synchronizing} and \emph{separating}) lying between
primitive and $2$-homogeneous are defined. A big open question is how the
permutation group induced by $S_n$ on $k$-subsets of $\{1,\ldots,n\}$ fits
in this hierarchy; our conjecture would give a solution to this problem for
$n$ large in terms of $k$. We prove the conjecture in the case $k=4$: our
result asserts that $S_n$ acting on $4$-sets is separating for $n\ge10$
(it fails to be synchronizing for $n=9$).

\end{abstract}

\section{The definitions}

The concepts of synchronization and separation were defined for permutation
groups about ten years ago (see~\cite{acs}). We outline the  definitions here
and extend them to association schemes.

A (finite-state deterministic) automaton is a machine which can be in one of a
finite set $\Omega$ of internal states. On reading a letter from the alphabet
associated with the automaton, it undergoes a change of state. If the machine
reads a word (a finite sequence of letters), the corresponding state changes
are composed.

An automaton can also be regarded as a directed graph whose arcs are labelled
with the letters in the alphabet, having the property that there is a unique
arc with each label leaving any vertex. (Loops and multiple arcs are allowed.)
Alternatively, it can be regarded as a transformation semigroup on the set
$\Omega$ with a distinguished set of generators (the transformations which
correspond to single letters). 

An automaton is said to be \emph{synchronizing} if there is a word in its
alphabet (called a \emph{reset word}) such that, after reading this word, the
automaton is in a known state, regardless of its starting state. In the
semigroup interpretation, an automaton is synchronizing if the semigroup
contains a transformation of rank~$1$ (one whose image consists of a single
element).

Clearly it is not possible for a permutation group to be synchronizing if
$|\Omega|>1$, since any composition of permutations is a permutation. So we
abuse terminology and say that a permutation group $G$ on $\Omega$ is
\emph{synchronizing} if, for any transformation $a$ of $\Omega$ which is not
a permutation, the semigroup $\langle G,a\rangle$ is a synchronizing semigroup.
(In other words, the automaton whose transitions are generators of $G$
together with an arbitrary non-permutation is synchronizing.)

There is a combinatorial characterisation as follows. A \emph{$G$-transversal}
for a partition $P$ of $\Omega$ is a set $A$, all of whose images under $G$ are
transversals for $P$. Now the group $G$ is non-synchronizing if there is a nontrivial
partition for which all images of some set $A$ are $G$-transversals. (The trivial partitions are the partition into singletons and the partition with a single part.)

This can also be formulated in terms of simple graphs (loopless and without 
multiple edges). The \emph{clique number} of a graph is the number of vertices
in a largest complete subgraph, and the \emph{chromatic number} is the least
number of colours required to colour the vertices so that the ends of any edge
receive different colours.
An \emph{endomorphism} of a graph is a transformation on the vertex set of the
graph which maps edges to edges. The endomorphisms of a graph $\Gamma$ form a
transformation semigroup $\End(\Gamma)$. The following theorem can be found in
\cite{acs}.

\begin{theorem}\leavevmode
\begin{enumerate}
\item
A transformation semigroup $S$ is non-synchronizing if and only if there exists
a non-null graph $\Gamma$ such that $S\le\End(\Gamma)$ and the clique number
and chromatic number of $\Gamma$ are equal.
\item
A permutation group $G$ is non-synchronizing if and only if there exists a
graph $\Gamma$, neither null nor complete, such that $G\le\Aut(\Gamma)$ and
the clique number and chromatic number of $\Gamma$ are equal.
\end{enumerate}
\end{theorem}

In terms of the previous characterisation, the partition of $\Omega$ is given
by a colouring of $\Gamma$ with the minimum number of colours, and a
$G$-transversal is a clique with size equal to this number.

A related concept is also defined in \cite{acs}. A transitive permutation group
$G$ on $\Omega$ is \emph{non-separating} if there exist subsets $A$ and $B$
of $\Omega$ with $|A|,|B|>1$ such that $|Ag\cap B|=1$ for all $g\in G$ (this
entails $|A|\cdot|B|=|\Omega|$); it is \emph{separating} otherwise. Again, it
is known that a permutation group $G$ is non-separating if and only if there is
a graph $\Gamma$, neither null nor complete, such that the product of its
clique number and its coclique number (the size of the largest coclique) is
$|\Omega|$, and with $G\le\Aut(\Gamma)$. (In the definition, we take $A$ and
$B$ to be a clique and a coclique of maximum size.)

It is easy to see that a separating group is synchronizing. (If $G$ is 
transitive and non-synchronizing, take $P$ to be a nontrivial partition and $A$
a $G$-transversal for $P$; then $|Ag\cap B|=1$ and $|A|\cdot|B|=|\Omega|$ for
every part $B$ of $P$.) The converse is false, but examples are not easy to
come by.

The following theorems are proved in \cite{acs}.

\begin{theorem}
Let $n\ge5$, and let $G$ be the permutation group induced by the symmetric
group $S_n$ on the set of $2$-element subsets of $\{1,\ldots,n\}$. Then the
following are equivalent:
\begin{enumerate}
\item $G$ is synchronizing;
\item $G$ is separating;
\item $n$ is odd.
\end{enumerate}
\end{theorem}

\begin{theorem}
Let $n\ge7$, and let $G$ be the permutation group induced by the symmetric
group $S_n$ on the set of $3$-element subsets of $\{1,\ldots,n\}$. Then the
following are equivalent:
\begin{enumerate}
\item $G$ is synchronizing;
\item $G$ is separating;
\item $n$ is congruent to $2$, $4$ or $5\pmod{6}$ and $n>8$.
\end{enumerate}
\end{theorem}

The aim of this paper is to extend these results to larger values of $k$; we
achieve a complete result for $k=4$ (using methods based on the work of
Delsarte on association schemes, which will hopefully extend to larger values
of $k$) and a conjecture about the situation when $n$ is sufficiently large
in terms of $k$.

\section{Steiner systems and the main conjecture}

An \emph{association scheme} is a collection $\{A_0,\ldots,A_s\}$ of symmetric
zero-one matrices, with $A_0=I$, whose sum is the all-$1$ matrix, and having
the property that the linear span (over the real numbers) of the matrices is
an algebra (closed under matrix multiplication), called the \emph{Bose--Mesner
algebra} of the scheme. See \cite{rab,delsarte} for further details. (Note that
Delsarte uses a slightly more general definition, but his important examples
fit the definition given here.)

Given an association scheme as above, each matrix $A_i$ for $i>0$ is the
adjacency matrix of a graph, as indeed are the sums of some of these matrices.
For $I\subseteq\{1,\ldots,s\}$, we denote by $\Gamma_I$ the graph whose
adjacency matrix is $\sum_{i\in I}A_i$. Such a graph will be called
\emph{non-trivial} if it is neither complete nor null, that is, if 
$I\ne\varnothing$ and $I\ne\{1,\ldots,s\}$. Note that the complement of the graph
$\Gamma_I$ is $\Gamma_{\{1,\ldots,s\}\setminus I}$. An important property of
graphs in an association scheme is that, in any such graph, the product of
the clique number and the coclique number is at most the number of
vertices. (They share this property with vertex-transitive graphs; neither
class of graphs includes the other.)

Adapting the definitions from permutation group theory given above, we say
that an association scheme is \emph{non-synchronizing} if there is a
non-trivial graph in the scheme with clique number equal to chromatic number,
and is \emph{synchronizing} otherwise; moreover, the scheme is
\emph{non-separating} if there is a non-trivial graph in the scheme such that
the product of its clique number and its coclique number is equal to the
number of vertices, and is \emph{separating} otherwise.

In this paper we are concerned with the \emph{Johnson scheme}. This is the
association scheme $J(n,k)$ whose vertices are the $k$-element subsets of
an $n$-element set (without loss of generality $\{1,\ldots,n\}$); for each
$i$ with $0\le i\le k$, we take the matrix with rows and columns indexed by
the vertices, having $(A,B)$ entry $1$ if $|A\cap B|=i$, and $0$ otherwise.
It is convenient to change the order of the indices by calling this matrix
$A_{k-i}$. Then $A_0$ is the identity matrix, and the other matrices in the
scheme are $A_1,\ldots,A_k$. For $I\subseteq\{1,\ldots,k\}$, we let
$A_I=\sum_{i\in I}A_i$, and let $\Gamma_I$ be the graph with adjacency matrix
$A_I$. 

While this numbering seems a little odd, our notation is chosen to agree with
that of Delsarte~\cite{delsarte}, who considered this scheme in detail and
gave the eigenvalues of the matrices in the scheme in terms of Eberlein
polynomials, as we will discuss.

Two particular cases will be very important, so we introduce a different
notation for them:
\begin{itemize}\itemsep0pt
\item $\Delta_t(n,k)=\Gamma_{\{k-t+1,\ldots,k\}}(n,k)$, the graph in which two
$k$-sets are joined if they intersect in fewer than $t$ points;
\item $\Phi_t(n,k)=\Gamma_{\{1,\ldots,k-t\}}(n,k)$, the complement of
$\Delta_t(n,k)$, in which two $k$-sets are joined if they intersect in at
least $t$ points.
\end{itemize}
\medskip

A \emph{Steiner system} $S(t,k,n)$, for $0<t<k<n$, is a collection
$\mathcal{B}$ of $k$-subsets of $\{1,\ldots,n\}$ with the property that any
$t$-subset of $\{1,\ldots,n\}$ is contained in a unique member of
$\mathcal{B}$. It is well-known that, for $0\le i\le t-1$, the number of
members of $\mathcal{B}$ containing an $i$-subset of $\{1,\ldots,n\}$ is
$\binom{n-i}{t-i}/\binom{k-i}{t-i}$, independent of the choice of $i$-set.
Thus necessary conditions for the existence of $S(t,k,n)$ are that
\[\binom{k-i}{t-i}\hbox{ divides }\binom{n-i}{t-i}\hbox{ for }0\le i\le t-1.\]
We refer to these as the \emph{divisibility conditions}. In a remarkable recent
result, Keevash~\cite{keevash} showed:

\begin{theorem}
There exists a function $F$ such that, if $n\ge F(t,k)$ and the divisibility
conditions are satisfied, then a Steiner system $S(t,k,n)$ exists.
\end{theorem}

The connection with what went before is that the set of blocks of a Steiner
system $S(t,k,n)$ is a clique in the graph $\Delta_t(n,k)$, of size
$\binom{n}{t}/\binom{k}{t}$. Moreover, there is a coclique in this
graph of size $\binom{n-t}{k-t}$, consisting of all the $k$-sets containing
a fixed $t$-set. (We say that such a cocloque is of \emph{EKR type}, after
the theorem of Erd\H{o}s, Ko and Rado, asserting that they are the largest
cocliques provided that $n$ is sufficiently large.) 

Moreover, it is easily checked that
\[\binom{n-t}{k-t}\cdot\binom{n}{t}\Big/\binom{k}{t}=\binom{n}{k}.\]
So, if a Steiner system $S(t,k,n)$ exists, then the product of the clique
number and coclique number in $\Delta_t(n,k)$ is equal to the
number of vertices, and the Johnson scheme $J(n,k)$ is non-separating.

Our main conjecture is that, asymptotically, the converse holds:

\begin{conj}
There is a function $G$ such that, if $n\ge G(k)$ and the Johnson scheme
$J(n,k)$ is non-separating, then there exists a Steiner system $S(t,k,n)$
for some $t$ with $0<t<k$.\label{c1}
\end{conj}

Putting this conjecture together with Keevash's theorem, we can re-formulate
it as follows:

\begin{conj}
There is a function $H$ such that, if $n\ge H(k)$, then the Johnson scheme is
non-separating if and only if the divisibility conditions for $S(t,k,n)$
are satisfied for some $t$ with $0<t<k$.\label{c2}
\end{conj}

What about synchronization? We first observe that, for sufficiently large $n$,
the graph $\Delta_t(k,n)$ cannot have clique number equal to chromatic number.
First we note the exact bound in the Erd\H{o}s--Ko--Rado theorem,
proved by Wilson~\cite{wilson}:

\begin{theorem}
Suppose that $n>(t+1)(k-t+1)$. Then a coclique in the graph
$\Delta_t(n,k)$ has size at most $\binom{n-t}{k-t}$, with equality if
and only if it is of EKR type.
\end{theorem}

\begin{theorem}
For $t<k$, there is no partition of the set of $k$-subsets of an $n$-set
into cocliques of EKR type in the graph $\Delta_t(n,k)$.
\label{t:ekr_col}
\end{theorem}

\begin{proof}
First we note that, if $k\ge2t$, then any two $t$-sets are contained in some
$k$-set; so sets of EKR type intersect. So we may assume that $k<2t$.

A set $S$ of EKR type consists of all the $k$-sets containing a fixed $t$-set
$T$, which we call its \emph{kernel}; we denote $S$ by $S(T)$. Now, if
$T_1$ and $T_2$ are $t$-sets with $|T_1\cap T_2|\ge 2t-k$, then
$|T_1\cup T_2|\le k$, and so $S(T_1)\cap S(T_2)\ne\varnothing$. So, in a family
of pairwise disjoint sets of EKR type, the kernels are $t$-sets which intersect
in at most $2t-k-1$ points.

The number of kernels in such a collection is at most
$\binom{n}{2t-k}/\binom{t}{2t-k}$; so we are done if we can show that
\[\binom{n}{t}\Big/\binom{k}{t}>\binom{n}{2t-k}\Big/\binom{t}{2t-k}.\]

For this, it is enough to show that the following claims are true:
\begin{itemize}
\item[(i)] $\displaystyle{\binom{n}{2t-k}{\frac{t! (n-2t+k)!}{k! (n-t)!}}
= \binom{n}{t}\binom{t}{2t-k}\Big/\binom{k}{t}}$;
\item[(ii)] $(t! (n-2t+k)!)/(k! (n-t)!) > 1$.
\end{itemize}   
For (i), we have
\begin{eqnarray*}
\binom{n}{t}\binom{t}{2t-k}\Big/\binom{k}{t}
&=& \frac{n!}{(n-2t+k)! (2t-k)!} \frac{t! (n-2t+k)!}{k! (n-t)! }\\
&=& \binom{n}{2t-k} {\frac{t! (n-2t+k)!}{k! (n-t)!}}.
\end{eqnarray*}
For (ii),
\[\frac{t! (n-2t+k)!}{k! (n-t)!}= \frac{(n-2t+k)(n-2t+k-1)\cdots (n-2t+k-(k-t)+1)}{k(k-1)\cdots (k-(k-t)+1)}\]
And since $n-2t+k> k$ for $n\geq 2k$ we have that
\[\frac{t! (n-2t+k)!}{k! (n-t)!} > 1.\qedhere \]
\end{proof}

We also have to consider the possibility that the complementary graph
$\Phi_t(n,k)$ has clique number equal to chromatic number.

The existence of cliques of size
$\binom{n-t}{k-t}$ in $\Phi_t(n,k)$ shows that the size of a coclique in
this graph is at most $\binom{n}{t}/\binom{k}{t}$, a fact that is easily
proved directly; equality holds if and only if the coclique is a Steiner
system. So the graph has clique number equal to chromatic number if and
only if the set of $k$-subsets of an $n$-set can be partitioned into block
sets of Steiner systems $S(t,k,n)$. A collection of Steiner systems which
partitions the set of all $k$-subsets is a \emph{large set}. In view of this,
we further conjecture the following:

\begin{conj}
There is a function $L$ such that, if $n\ge L(k)$, then the Johnson scheme
$J(n,k)$ is non-synchronizing if and only if there exists a large set of
Steiner triple systems $S(t,k,n)$ for some $t$ with $0<t<k$.
\end{conj}

Much less is known about the existence of large sets. The main results are
the following:
\begin{enumerate}
\item For $t=1$, an $S(t,k,n)$ is simply a partition of $\{1,\ldots,n\}$ into
sets of size $k$, which exists if and only if $k$ divides $n$. A theorem of
Baranyai~\cite{baranyai} shows that a large set of such partitions exists
whenever $k$ divides $n$.
\item For $t=2$, $k=3$, Kirkman~\cite{kirkman} showed that a $S(2,3,n)$
exists if and only if $n\equiv1$ or $3\pmod{6}$. The smallest example is
the \emph{Fano plane} with $n=7$. Cayley~\cite{cayley} showed that there does
not exist a large set of Fano planes; indeed, there do not exist more than two
pairwise disjoint Fano planes. However, Lu and Teirlinck~\cite{teirlinck}
showed that, for all ``admissible'' $n$ greater than $7$, a large set of
Steiner triple systems exists.
\item Koloto\u{g}lu and Magliveras~\cite{km} have constructued large sets of
projective planes of orders $3$ and $4$ (that is, $S(2,4,13)$ and $S(2,5,21)$).
\end{enumerate}

It may be that large sets of $S(t,k,n)$ exist whenever the divisibility 
conditions are satisfied and $n$ is sufficiently large. If so, then our
conjecture above would imply that, for $n$ sufficiently large in terms of
$k$, the Johnson scheme $J(n,k)$ is synchronizing if and only if it is
separating. However, we are not sufficiently confident to conjecture this!

\section{Small examples}
\label{s3}

An association scheme is non-separating if there is a graph in the scheme for
which the product of the clique and coclique numbers is equal to the number
of vertices. Given such a situation, there are two ways that synchronization
can fail: either there is a partition of the vertices into cocliques of maximal
size, so that the clique number and chromatic number are equal, or there is
a partition into cliques of maximal size, so that this equality holds in the
complementary graph.

As we saw in the preceding section, if $n>(t+1)(k-t+1)$, then $\Delta_t(n,k)$
cannot have clique number equal to chromatic number, since the maximum-size
cocliques are of EKR type; and complement $\Phi_t(n,k)$ has clique number equal
to chromatic number if and only if there is a large set of Steiner systems
$S(t,k,n)$.

For smaller $n$, non-synchronization can arise in one of two ways: either
because there are other large cocliques in $\Delta_t(n,k)$, or because one
of the other graphs $\Gamma_I(n,k)$ in the association scheme has clique
number equal to chromatic number.

Here are five examples. In all but one (the case $n=9$, $k=4$) it is the first
of these two ways which occurs.

\paragraph{The case $k=3$, $n=7$.} In this case, for $t=2$, a clique is the
block set of the $S(2,3,7)$, the Fano plane. Since it is a projective plane,
any two lines meet in a point, so it is a clique in the graph
$\Gamma_{\{2\}}(3,7)$.
As well as cocliques of EKR type (all $3$-sets containing a given two points),
there are cocliques of size $5$ defined as follows: let $L$ be a line of the
Fano plane, and take $L$ together with the four $3$-sets disjoint from it.
Furthermore, the seven such sets obtained by performing this construction for
each line of the Fano plane partition the $35$ sets of size $3$ into seven
cocliques of size $5$. So this graph has clique number equal to
chromatic number.

\paragraph{The case $k=3$, $n=8$.} Again a Fano plane gives us a $7$-clique
in the graph $\Gamma_{\{2\}}(3,8)$. Now the eight $3$-sets consisting of a
line $L$, three sets each comprising two points of $L$ and the point outside
the Fano plane, and the four sets consisting of three of the four points of
the Fano plane outside $L$, form a coclique; doing this for the seven lines we 
obtain a partition of the $56$ sets of size $3$ into seven cocliques of
size $8$, so this graph has clique number equal to chromatic number. Another
way of viewing this is to observe that the Fano plane has an extension to a
$S(3,4,8)$ whose blocks fall into $7$ parallel classes with two blocks in 
each; the eight $3$-sets contained in a block of a parallel class form
a coclique, and we obtain seven such sets, one for each parallel class.

\paragraph{The case $k=4$, $n=9$.} The Steiner system $S(3,4,8)$ has $14$
blocks, any two meeting in $0$ or $2$ points. We can construct a set of $9$
subsets of size $4$, any two meeting in $1$ or $3$ points, as follows:
partition $\{1,\ldots,9\}$ into three sets of size $3$, arranged around a
circle; now take the $4$-subsets consisting of one part and a single point
of the next (in the cyclic order).

Breach and Street~\cite{bs} showed that the $126$ $4$-subsets of a $9$-set
can be partitioned into a so-called \emph{overlarge set} of nine Steiner
systems $S(3,4,8)$ (each omitting a point); indeed, this can be done in just
two non-isomorphic ways, each admitting a $2$-transitive group. This gives a
colouring of the graph $\Gamma_{\{1,3\}}(4,9)$ corresponding to
intersections $1$ and $3$. (Their proof was computational; a more geometric
proof involving triality was given by Cameron and Praeger~\cite{cp}.)

\paragraph{The case $k=5$, $n=11$.} An example similar to the first is obtained
using the Steiner system $S(4,5,11)$, any two of whose blocks intersect in $1$,
$2$ or $3$ points. Thus the blocks form a $66$-clique in the graph
$\Gamma_{\{2,3,4\}}(11,5)$. A block together with the six $5$-sets disjoint
from it form a coclique of size $7$,
and the $66$ sets obtained in this way form a colouring of the graph.

\paragraph{The case $k=5$, $n=12$.} Again the blocks of $S(4,5,11)$ form a
$66$-clique in $\Gamma_{\{2,3,4\}}(12,5)$. The Steiner system has an extension
to a $S(5,6,12)$ whose blocks come in $66$ parallel classes with two disjoint
blocks in each, and the twelve $5$-sets contained in a block of a fixed
parallel class form a coclique. 

\section{The case $I=\{1,k\}$}

In this section, we deal with the case $I=\{1,k\}$ (or the complement
$I=\{2,\ldots,k-1\}$) of Conjecture~\ref{c1}, and show that these cannot
occur if $n$ is sufficiently large. In other words, taking account of the
indexing used in the Johnson scheme, we show the following.

\begin{theorem}
There is a function $f$ such that, if $n\ge f(k)$, and $S$ and $T$ are families
of $k$-subsets of $\{1,\ldots,n\}$ with the property that $S$ is
$\{0,k-1\}$-intersecting (that is, any two of its members intersect in $0$ or
$k-1$ points) and $T$ is $\{1,\ldots,k-2\}$-intersecting, then
$|S|\cdot|T|<{n\choose k}$.
\label{specialcase}
\end{theorem}

\begin{proof}
The proof proceeds in three steps.

\paragraph{Step 1} $|S|\le n$.

To see this, consider first a $(k-1)$-intersecting family $U$ of $k$-sets.
It is easy to see that there are just two possibilities:
\begin{enumerate}
\item all members of $U$ contain a fixed $(k-1)$-set;
\item all members of $U$ are contained in a fixed $(k+1)$-set.
\end{enumerate}

Next we claim that the relation $\sim$ on $S$ defined by $A\sim B$ if $A=B$ or
$|A\cap B|=k-1$ is an equivalence relation. It is clearly reflexive and
symmetric, so suppose that $A\sim B$ and $B\sim C$. Then
$|A\cup B|=|B\cup C|=k+1$, and so $|A\cap C|\ge k-2$, whence $|A\cap C|=k-1$ as
required.

Now if two members of $S$ belong to distinct equivalence classes, they are
disjoint. So the support of $S$ (the set of points lying in some element of
$S$) is the union of the supports of the equivalence classes, which are
pairwise disjoint. We have seen that the number of sets in each equivalence
class does not exceed the cardinality of its support; so the same holds for
$S$, and the claimed inequality follows.

For the next step, we note that $T$ is an intersecting family. We split the
proof into two subcases.

\paragraph{Step 2} If the intersection of the sets in $T$ is non-empty, then
$|T|\le{n-1\choose k-2}/(k-1)$.

For let $x$ be the unique point in the intersection. Then
\[T=\{\{x\}\cup B:B\in T'\},\]
where $T'$ is a $\{0,\ldots,k-3\}$-intersecting family of $(k-1)$-subsets of
$\{1,\ldots,n\}\setminus\{x\}$; in other words, a partial $S(k-2,k-1,n-1)$.
So $|T|=|T'|\le{n-1\choose k-2}/(k-1)$, the right-hand side being the number
of blocks in a hypothetical Steiner system with these parameters.

\paragraph{Step 3} If the intersection of the sets in $T$ is empty, then
$|T|\le{n-1\choose k-1}-{n-k-1\choose k-1}+1$.

Since $T$ is an intersecting family, this is just the conclusion of the
Hilton--Milner theorem~\cite{hm}.

\paragraph{Conclusion of the proof} We have
\[{n\choose k}=|S|\cdot|T|\le
\begin{cases}
n{n-1\choose k-2}/(k-1) & \mbox{ if }\bigcap T\ne\varnothing,\\
n({n-1\choose k-1}-{n-k-1\choose k-1}+1) & \mbox{ if }\bigcap T=\varnothing.
\end{cases}\]
But in each case, for fixed $k$, the left-hand side of the inequality is
a polynomial of degree $k$ in $n$, whereas the right-hand side is a polynomial
of degree $k-1$; thus the inequality holds for only finitely many values of $n$.
\end{proof}

We remark that, in fact, we know of no examples meeting the bound for this
case with $n>2k$. So as well as extending these techniques to other cases,
the problem of deciding whether the bound is always strict remains.

In the next section we use the fact proved in the first part of the above proof:

\begin{cor}
For $k\ge3$ and $n>2k$, a $\{0,k-1\}$-intersecting family of $k$-subsets of
$\{1,\ldots,n\}$ has size at most $n$.
\label{nbound}
\end{cor}

Indeed, it is not too hard to find the precise upper bound; but we do not
require this.

\section{The case $k=4$}
\label{kis4}

\subsection{Background}

To handle the case $k=4$, we use the results of Delsarte~\cite{delsarte} on
association schemes. We begin with a brief introduction to this material, but
we also refer to  \cite[Chapter 6]{gm} where some of the computations 
that we omit are done in detail.

Suppose that the matrices $A_0,\ldots,A_{k-1}$ (with $A_0=I$) span the
Bose--Mesner algebra of an association scheme on $v$ points. Since this algebra
is a commutative algebra of real symmetric matrices, the matrices are simultaneously
diagonalisable: that is, there are idempotent matrices $E_0,\ldots,E_{k-1}$
spanning the same algebra, with $E_0=J/v$, where $J$ is the all-$1$ matrix.
Thus, for some coefficients $P_k(i)$ and $Q_k(i)$, we have
\begin{eqnarray*}
A_j &=& \sum_{i=0}^{k-1}P_j(i)E_i,\\
E_j &=& v^{-1}\sum_{i=0}^{k-1}Q_j(i)A_i.
\end{eqnarray*}
Here the numbers $P_j(i)$ for $i=0,\ldots,k-1$ are the eigenvalues of $A_j$.
The matrices with $(i,j)$ entry $P_j(i)$ and $Q_j(i)$ are called the
\emph{matrix of eigenvalues} and \emph{dual matrix of eigenvalues} of the
scheme.

Delsarte used these matrices to provide bounds on the clique and coclique
numbers of graphs in an association scheme:

\begin{theorem}[\cite{delsarte}, Theorem 5.9; see also \cite{gm}]
Let $\mathcal{A}$ be an association scheme on $v$ vertices and let $\Gamma$
 be the union of some of the graphs in the scheme. If $C$ is a clique and $S$ is an coclique in $\Gamma$,
 then $|C|\cdot|S|\le v$. If equality holds and $x$ and $y$ are the respective characteristic vectors of $C$ and $S$, then
 $(xE_jx^\top)(yE_jy^\top)=0$ for all $j>0$.
\end{theorem}

In the above notation, the inner distribution $a$ of $C$ is the vector where $a_i=x A_i x^\top / |C|$
for each $i\in\{0,\ldots, d\}$ (and $\mathcal{A}$ has $d$ classes).
Now if $Q$ is the dual matrix of eigenvalues of $\mathcal{A}$, then
\[
(a Q)_j = \frac{v}{|C|} x E_j x^\top
\]
for all $j\ge 0$. This vector is sometimes known as the \emph{MacWilliams transform} of $C$.

The \emph{degree set} of a subset $X$ of the vertices of $\Gamma$ is the set of nonzero indices $i$
for which the $i$-th coordinate of its inner distribution is nonzero.
The \emph{dual degree set} of $X$ is the set of nonzero indices $j$
for which the $j$-th coordinate of its MacWilliams transform is nonzero.
Two subsets $X$ and $Y$ of the vertices of $\Gamma$ are \emph{design-orthogonal} if
their dual degree sets are disjoint. Similarly, $X$ and $Y$ are \emph{code-orthogonal} if
their degree sets are disjoint. 

\begin{cor}\label{Delsarte_cor}
Suppose a clique $C$ and coclique $S$ meet the bound; $|C|\cdot|S|=v$.
 Then the Schur product of the MacWilliams transforms of $C$ and $S$ equals
$(v,0,\ldots,0)$.
\end{cor}

\begin{cor}\label{cor:ratiobound}
Suppose $\Gamma$ is a graph from an association scheme $\mathcal{A}$ on $v$
vertices, and suppose $\alpha(\Gamma)\omega(\Gamma)=v$. Then
$\omega(\Gamma)=1-\deg(\Gamma)/\tau$,
where $\tau$ is the smallest eigenvalue of $\Gamma$,
and in particular, $\tau$ divides $\deg(\Gamma)$.
\end{cor}

\begin{proof}
Let $\tau$ be the smallest eigenvalue of $\Gamma$.
By the result of Lov\'asz~\cite{lovasz},
$\alpha(\Gamma) \le \frac{v}{1-\deg(\Gamma)/\tau}$.
So since $\alpha(\Gamma)\omega(\Gamma)=v$, we have  
$v\le \omega(\Gamma) v / (1-\deg(\Gamma)/\tau)$ and hence $
1-\deg(\Gamma)/\tau \le \omega(\Gamma)$;  that is, we obtain equality 
$\omega(\Gamma)=1-\deg(\Gamma)/\tau$.
\end{proof}

We can find a simple expression for the eigenvalues of the Johnson scheme from
Section 4.2.1 of Delsarte's PhD thesis. The analogue of the Krawchouk
polynomials of the Hamming scheme to the Johnson scheme are the \emph{Eberlein
polynomials}, or \emph{dual Hahn polynomials}.
Given an integer $0\le j\le k$, we define the Eberlein polynomial $E_j(x)$ in the indeterminate $x$, as follows:
\[
E_j(x):=\sum_{t=0}^j (-1)^{j-t}\binom{k-t}{j-t}\binom{k-x}{t}\binom{n-k+t-x}{t}.
\]
The $(i,j)$-entry of the matrix of eigenvalues of $P$ will be denoted $P_j(i)$.
Now Theorem 4.6 of Delsarte~\cite{delsarte} asserts:

\begin{theorem}
The matrix $P$ of eigenvalues and the dual matrix of eigenvalues $Q$ of the Johnson scheme
$J(n,k)$ are given by
\[
P_j(i)=E_j(i), \quad Q_i(j)=\frac{\binom{n}{i} -\binom{n}{i-1} }{\binom{k}{j}\binom{n-k}{j}} E_j(i),
\]
for $i,j=0,1,\ldots,k$.
\end{theorem}

\subsection{The main result}

%
%

\begin{theorem}
Let $n\ge 10$ and $I\in \{ \{1,3,4\}, \{1,3\}, \{1,4\}, \{1,2,4\} \}$, and let $\Delta:=\Gamma(n,4,I)$.
Then $\omega(\Delta)\alpha(\Delta) < |V\Delta|=\binom{n}{4}$.
\end{theorem}

\begin{proof}
A straight-forward calculation shows that the matrix $P$ of eigenvalues, and the dual matrix $Q$, 
of the Johnson scheme $J(n,4)$ are
{\footnotesize
\begin{align*}
P&=
\begin{pmatrix}
 1 & 4 (n-4) & 3 (n-5) (n-4) & \frac{2}{3} (n-6) (n-5) (n-4) & \frac{1}{24} (n-7) (n-6) (n-5) (n-4) \\
 1 & 3 n-16 & \frac{3}{2} (n-8) (n-5) & \frac{1}{6} (n-16) (n-6) (n-5) & -\frac{1}{6} (n-7) (n-6) (n-5) \\
 1 & 2 (n-7) & \frac{1}{2} ((n-21) n+92) & -(n-9) (n-6) & \frac{1}{2} (n-7) (n-6) \\
 1 & n-10 & -3 (n-8) & 3 n-22 & 7-n \\
 1 & -4 & 6 & -4 & 1 \\
\end{pmatrix},\\
Q&=
\begin{pmatrix}
 1 & n-1 & \frac{1}{2} (n-3) n & \frac{1}{6} (n-5) (n-1) n & \frac{1}{24} (n-7) (n-2) (n-1) n \\
 1 & \frac{1}{4}(3n-7)-\frac{3}{n-4} & \frac{(n-7) (n-3) n}{4 (n-4)} & \frac{(n-10) (n-5) (n-1) n}{24 (n-4)} & -\frac{(n-7) (n-2) (n-1) n}{24 (n-4)} \\
 1 & \frac{(n-8) (n-1)}{2 (n-4)} & \frac{(n-3) n ((n-21) n+92)}{12 (n-5) (n-4)} & -\frac{(n-8) (n-1) n}{6 (n-4)} & \frac{(n-7) (n-2) (n-1) n}{12 (n-5) (n-4)} \\
 1 & \frac{(n-16) (n-1)}{4 (n-4)} & -\frac{3 (n-9) (n-3) n}{4 (n-5) (n-4)} & \frac{(n-1) n (3 n-22)}{4 (n-6) (n-4)} & -\frac{(n-7) (n-2) (n-1) n}{4 (n-6) (n-5) (n-4)} \\
 1 & -\frac{4 (n-1)}{n-4} & \frac{6 (n-3) n}{(n-5) (n-4)} & -\frac{4 (n-1) n}{(n-6) (n-4)} & \frac{(n-2) (n-1) n}{(n-6) (n-5) (n-4)} \\
\end{pmatrix}.
\end{align*}
}

\begin{description}
\item[Case $I=\{1,3,4\}$:] 
The complement $\overline{\Delta}$ of $\Delta$ is a graph of an association scheme (because it is $\Gamma(n,4,\{2\})$.
Consider the third column of $P$.
The eigenvalues of $\overline{\Delta}$ are
\[
3 (n-5) (n-4),\frac{3}{2} (n-8) (n-5),\frac{1}{2} (n^2-21n+92),24-3n,6.
\]
The degree of $\overline{\Delta}$ is the first eigenvalue $3(n-5)(n-4)$. Most of the time, the fourth eigenvalue $24-3n$ is the smallest
eigenvalue, certainly when $n\ge 11$, when it is less than the third eigenvalue. So let us assume that $n\ge 11$.
By Corollary \ref{cor:ratiobound}, $24-3n$ divides $3 (n-5) (n-4)$, and hence $8-n$ divides $(n-5)(n-4)$.
So by elementary number theory, we have $n\in\{11,12,14,20\}$.

The cases $n=9, 10,11,12,14,20$ can be settled
each in turn. First, for $n=9$, the minimum eigenvalue of $\overline{\Delta}$ is $-8$, and the degree is $65$.
So by Corollary \ref{cor:ratiobound}, $\omega(\Delta)\alpha(\Delta) <\binom{n}{4}$.
For the remaining cases, we can calculate $\omega(\overline{\Delta})$ with
GAP/GRAPE \cite{gap,grape}\footnote{The share package GRAPE for the computer
algebra system GAP contains an efficient clique finder.} and compare it
to $1-\deg(\overline{\Delta})/\tau$, where $\tau$ is the smallest eigenvalue.
\begin{center}
\begin{tabular}{ccccc}
$n$ & $\omega(\overline{\Delta})$ & $1-\deg(\overline{\Delta})/\tau$\\ 
\hline
10 & 5 & 11 \\ 
11 & 6 & 15 \\ 
12 & 9 & 15 \\ 
14 & 13 & 16 \\ 
20 & 13 & 21 \\ 
\hline
\end{tabular}
\end{center}
We find that $\omega(\overline{\Delta})$ is never equal to $1-\deg(\overline{\Delta})/\tau$ and so 
by Corollary \ref{cor:ratiobound}, $\omega(\Delta)\alpha(\Delta) <\binom{n}{4}$.

\item[Case $I=\{1,2,4\}$:]  
This time, the complement graph $\overline{\Delta}$ is $\Gamma(n,4,\{3\})$.
The eigenvalues of $\overline{\Delta}$ are
\[
\frac{2}{3} (n-6) (n-5) (n-4),\frac{1}{6} (n-16) (n-6) (n-5),-(n-9) (n-6),3 n-22,-4
\]
For $n\ge 13$, the third eigenvalue of $\overline{\Delta}$ is the smallest in the spectrum of $\overline{\Delta}$.
So by Corollary \ref{cor:ratiobound}, $n-9$ divides $\frac{2}{3} (n-5) (n-4)$
and hence $n\in\{13,14,17,19,29,49\}$ (by basic elementary number theory).
For $n\in\{9,11\}$, the smallest eigenvalue of $\overline{\Delta}$ does not divide the degree.
Thus by Corollary \ref{cor:ratiobound}, we are left with $n\in\{10,12,13,14,17,19,29,49\}$ to check.
The cases $n=10,12,13,14,17,19,29,49$ can be settled
each in turn with the computer algebra system GAP/GRAPE:
\begin{center}
\begin{tabular}{cccc}
$n$ & $\omega(\overline{\Delta})$ & $1-\deg(\overline{\Delta})/\tau$\\ 
\hline
10 & 2 & 5\\
12 & 3& 9\\
13 & 3& 13\\
14 & 3& 13\\
17 & 4& 14\\
19 & 4& 15\\
29 & 7& 21\\
49 & 12 & 34\\
\hline
\end{tabular}
\end{center}

\item[Case $I=\{1,3\}$:]

Let $Q$ be the dual matrix of eigenvalues. The second column of $Q$ is
\[
c=\frac{n-1}{4(n-4)}\left(4(n-4), 3 n-16,2(n-8),n-16,-16 \right).
\]

Let $u=(1,a,0,x-a-1,0)$ and $v=(1,0,b,0,y-b-1)$ such that $a,b\ge 0$, $x-a-1,y-b-1\ge 0$ and $xy={n \choose 4}$.
So $u$ and $v$ are the inner distributions of an arbitrary clique and coclique of $\Gamma_I$, respectively,
attaining the clique-coclique bound.
The second entries of $uQ$ and $vQ$ are 
\begin{align*}
(uQ)_1=uc^\top=&\frac{n-1}{4(n-4)} ((2 a+x+3)n-16 x)\\
(vQ)_1=vc^\top=&\frac{n-1}{4(n-4)} 2((b+2) n-8 y).
\end{align*}
By Corollary \ref{Delsarte_cor}, the product of these quantities is zero, which gives us
 two scenarios: $x=n\frac{2 a +3 }{16-n}$ or $y=n\frac{1}{8} (b+2)$.
We will consider the former, with  the additional assumption that $xy={n\choose 4}$.
This then yields an expression for $y$:
\[
y=
-\frac{(n-16) (n-3) (n-2) (n-1)}{48 a+72}.
\]
Now $a,y\ge 0$, which implies that $n\le 16$. So we will assume for the moment that $n\ge 16$.
Therefore, 
\[
x=\frac{(n-3) (n-2) (n-1)}{3 (b+2)},\quad y=\frac{1}{8} (b+2) n.
\]
We now consider the two equations $(uQ)_2(vQ)_2=0$ and $(uQ)_4(vQ)_4=0$,
with the above values of $x$ and $y$ substituted in:
\begin{align*}
(b (n-10)+6 (n-4)) (a (b+2) (n-8)+b (2 n-13)-(n-4) ((n-10) n+7))&=0\\
a (b+2) (n-8)+18 (n-4)-b (n^2-12 n+38)&=0.
\end{align*}
The second equation gives us a value for $a$, which we can substitute into the first equation.
This results in the following equation:
\[
(n-5) (b (n-10)+6 (n-4)) (b-n+4)=0.
\]
However, $n\ne 5$ and $b (n-10)+6 (n-4)\ne 0$ (as $b\ge 0$ and $n>10$), so it follows that
$b=n-4$ and hence
\[
x=\frac{(n-3)(n-1)}{3},\quad y=\frac{1}{8} (n-2) n.
\]
Now suppose we have a $\{1,3\}$-clique $S$ of size $x=(n-1)(n-3)/3$.
Consider the members of $S$ containing a point $p$, with $p$ removed from each. 
This is a family of $3$-sets of an $(n-1)$-set, any two intersecting in $0$ or $2$ elements. 
By Corollary~\ref{nbound}, we know that there are at most $n-1$ of them. 
Now the standard double count gives
\[
|S| \le n(n-1)/4.
\]
But this is smaller than $(n-1)(n-3)/3$ so long as $n$ is at least $16$. 
This leaves the cases $10\le n\le 16$ to be considered (since we have excluded
$n=9$).

\begin{center}
\begin{tabular}{cccc}
$n$ & $\omega(\Gamma_I)$ & $\alpha(\Gamma_I)$\\
\hline
10 & 9 & 14 \\
11 & 9 & 14 \\
12 & 9 & 15\\
13 & 9 & 15\\
13 & 13 & 15 \\
14 & 13 & 21 \\
15 & 13 & 21 \\
16 & 13 & 28 \\
\hline
\end{tabular}
\end{center}

\item[Case $I=\{1,4\}$:] In this case, Theorem~\ref{specialcase} shows that
the result holds for sufficiently large $n$; indeed, the proof there works for
$n\ge46$. However, the intervening values are far too large for computation,
so we use the Q-matrix methods.

Let $u=(1,a,0,0,x-a-1)$ and $v=(1,0,b,y-b-1,0)$ such that $a,b\ge 0$, $x-a-1,y-b-1\ge 0$ and $xy={n\choose 4}$.
So $u$ and $v$ are the inner distributions of an arbitrary clique and coclique of $\Gamma_I$, respectively,
attaining the clique-coclique bound.
The second coordinates of $uQ$ and $vQ$ are 
\begin{align*}
(uQ)_1=uc^\top=&\frac{(n-1)}{4 (n-4)} ((3 a+4) n-16 x),\\
(vQ)_1=vc^\top=&\frac{(n-1)}{4 (n-4)} ( (b+y+3)n-16 y).
\end{align*}
By Corollary \ref{Delsarte_cor}, the product of these quantities is zero, which gives us 
two scenarios:
\begin{enumerate}
\item[(i)]  $x=n\frac{3a+4}{16}$, or 
\item[(ii)] $y=n\frac{b+3}{16-n}$.
\end{enumerate}
Since $y\ge 0$, Case (ii) does not arise if we assume $n\ge 17$, which we will for now. So suppose we have Case (i).
If we now consider the equation $(uQ)_2(vQ)_2=0$, we have
\begin{equation}\label{eq2}
((n-1) (a (n-11)+2 (n-8))+24 x) ((n-1) (b (n-11)+6 n-39)-9 (n-9) y)=0.
\end{equation}
Assuming (i), that is $x=n(3a+4)/16$, Equation \eqref{eq2} becomes
\begin{equation}\label{eq3}
(2 (a+2) n-11 a-16) \left((n-1) (b (n-11)+6 n-39)-9 (n-9) y\right)=0.
\end{equation}
However, if the first term is zero, then
$a = (16 - 4 n)/(2n-11)$
which is negative for $n\ge 6$; a contradiction.
Therefore, the second term in Equation \eqref{eq3} is zero.
In other words, we have an expression for $b$ in terms of $n$ and $y$:
\[
b=\frac{-6 n^2+9 n y+45 n-81 y-39}{(n-11) (n-1)}.
\]
If we now consider the equation $(uQ)_3(vQ)_3=0$, upon substitution of our value for $b$, we 
obtain an equation relating $a$, $y$ and $n$:
{\footnotesize
\begin{equation}\label{3point5}
(a (n^2-18n+68)+4 (n-8) (n-4)) (2 n^4-34 n^3-9 n^2 y+181 n^2+123 n y-335 n-426 y+186)=0.
\end{equation}
}
Having the first term in Equation \eqref{3point5} equal to zero leads to a contradiction, since in this case we would have
\[
a=-\frac{4 \left(n^2-12 n+32\right)}{n^2-18 n+68}
\]
which is negative for $n\ge 13$. So let us now assume $n\ge 13$. Then 
the second term in Equation \eqref{3point5} is zero, which gives us
an expression for $y$ in terms of $n$ (and we also obtain an expression for $x$):
\begin{align*}
x&=\frac{n(n-3)  ( 3n^2 -41n+142)}{8 (2n^2 - 28n+93)},\\
y&=\frac{(n-1) (n-2) (2n^2 - 28n+93)}{9n^2- 123 n +426}.
\end{align*}
However, $x\le n$ (Corollary~\ref{nbound}), which implies that $n\le 8$; a contradiction.


This leaves us now to consider by computer the cases where $9\le n\le 16$.
In each case, we see that the clique number and coclique number have a product
that does not attain ${n\choose 4}$.

\begin{center}
\begin{tabular}{cccc}
$n$ & $\omega(\Gamma_I)$ & $\alpha(\Gamma_I)$\\
\hline
9 & 6 & 12\\
10 & 10& 15 \\
11 & 10& 15\\
12 & 10& 15\\
13 & 10& 15\\
14 & 11& 15\\
15 & 15& 15\\
16 & 15& 15\\
\hline
\end{tabular}
\end{center}
\end{description}
\end{proof}

\begin{cor}
For $k=4$, Conjectures~\ref{c1} and \ref{c2} hold for $n\ge10$.
\end{cor}

This follows from the main theorem together with well-known results of Hanani
on the existence of Steiner systems $S(t,4,n)$ \cite{hanani34,hanani24}. Note
that $n=9$ is a genuine exception, as we saw in Section~\ref{s3}.

We conclude that, if $n\ge10$, then the permutation group induced by $S_n$
on $4$-sets fails to be separating if and only if one of the following holds:
\begin{itemize}
\item $n\equiv0\pmod{4}$ (so that $S(1,4,n)$ exists),
\item $n\equiv1$ or $4\pmod{12}$ (so that $S(2,4,n)$ exists),
\item $n\equiv2$ or $4\pmod{6}$ (so that $S(3,4,n)$ exists);
\end{itemize}
summarising, $n\equiv 0,1,2,4,8,10\pmod{12}$. What about synchronization?
Wilson's theorem~\cite{wilson} ensures that maximal cocliques are of EKR type,
so $G$ is non-synchronizing if and only if a large set of Steiner systems
exists. Baranyai's theorem~\cite{baranyai} ensures that $G$ is not
synchronizing if $n\equiv0\pmod{4}$, but existence of large sets of Steiner
systems in the other cases is unresolved (except for $t=2$, $n=13$, \cite{km}).

In particular, the existence of a Steiner system $S(3,4,10)$ shows that the
symmetric group $S_{10}$ acting on $4$-sets is non-separating. However, it is
synchronizing. Our above results show that synchronization could only fail if
there were a large set of seven pairwise disjoint $S(3,4,10)$ systems. However,
Kramer and Mesner~\cite{km2} showed that there cannot be more than five such
systems.

Thus, for $k=4$, synchronization and separation fail to be equivalent for
$S_n$ on $k$-sets, unlike the cases $k=2$ and $k=3$ described earlier.

\section{Projective planes}

The constructions in Section~\ref{s3} involved the fact that, in certain
Steiner systems, certain cardinalities of block intersection do not occur.
There are relatively few examples of such systems: the only ones known are
projective planes, $S(3,4,8)$, $S(4,5,11)$, $S(5,6,12)$, $S(3,6,22)$,
$S(4,7,23)$ and $S(5,8,24)$.

A \emph{projective plane} of order $q$ is a Steiner system $S(2,q+1,q^2+q+1)$
for some integer $q>1$. Projective planes of all prime power orders exist,
and none are known for other orders.

A projective plane has the property that any two of its blocks meet in a point.
Hence it is a clique in either of the graphs $\Gamma_{\{{}>k-2\}}(n,k)$ or
$\Gamma_{k-1}(n,k)$, with $k=q+1$ and $n=q^2+q+1$. In the first of these graphs,
cocliques of maximum cardinality are of EKR type, and we have
\[\binom{q^2+q-1}{q-1}\cdot(q^2+q+1)=\binom{q^2+q+1}{q+1},\]
so non-separation holds for this graph: these sets (all $k$-sets containing
two given points) also show non-separation for $\Gamma_1(n,k)$. Also, by
Theorem~\ref{t:ekr_col}, we cannot partition the $k$-sets into subsets of
EKR type.

In the case $q=2$, in our example above, we observed that there were other
cocliques, so that the possibility of a colouring with $q^2+q+1$
colours cannot be ruled out; and indeed we saw that such a colouring exists.

\begin{conj}
For $q>2$, a coclique of maximum size in the graph $\Gamma_q(q^2+q+1,q+1)$
must consist of all the $(q+1)$-sets containing two given points; so the
chromatic number of this graph is strictly larger than $q^2+q+1$.
\end{conj}

A simple computation shows that the conjecture is true for $q=3$ and for $q=4$.

On the other hand, the truth of this conjecture would probably not give an
infinite family of examples which are synchronizing but not separating.
Magliveras conjectured that large sets of projective planes of any order $q>2$
exist; the existence is shown for $q=3$ and $q=4$ in \cite{km}.

\section{Conclusion}

The computations with rational functions in Section~\ref{kis4} were performed
with Mathematica~\cite{mathematica}; computations of clique number in special
cases were done with \textsf{GAP} and its package \texttt{GRAPE}, as already
noted.

The techniques we used to prove the separation conjecture for $k=4$ are
amenable to the use of traditional computer algebra systems such as
Mathematica, and it should be possible to settle several more values of $k$.

In the arguments for $I=\{1,3\}$ and $I=\{1,4\}$, we saw that the sizes of a
clique and a coclique whose product is equal to the number of vertices can be
determined from the Q-matrix of the association scheme. Is this true in general?
(It does hold for two-class association schemes, that is, strongly regular
graphs.)

The synchronization question is likely to be harder; general results on the
existence or nonexistence of large sets of Steiner systems are likely to 
require a significant new idea, and even particular cases involve quite large
computations with \textsf{GAP} or more specialised software. Perhaps new
techniques in hypergraph decomposition will help.

There seems plenty of scope for extending these results to other primitive
association schemes.


\end{document}